\documentclass[10pt, leqno]{amsart}
\usepackage{amsfonts,amssymb}
\allowdisplaybreaks[4]

\usepackage{amsmath}
\usepackage{amsthm}
\usepackage{amsrefs}
\usepackage{qsymbols}
\usepackage{latexsym}
\usepackage{chngcntr}
\usepackage[noadjust]{cite}
\usepackage{paralist}
\usepackage{fixltx2e}
\usepackage{esint}

\newtheorem{theorem}{Theorem}[section]

\newtheorem{lemma}[theorem]{Lemma}
\newtheorem{proposition}[theorem]{Proposition}

\theoremstyle{definition}

\newtheorem{remark}[theorem]{Remark}

\newtheorem{innercustomgeneric}{\customgenericname}
\providecommand{\customgenericname}{}
\newcommand{\newcustomtheorem}[2]{%
  \newenvironment{#1}[1]
  {%
   \renewcommand\customgenericname{#2}%
   \renewcommand\theinnercustomgeneric{##1}%
   \innercustomgeneric
  }
  {\endinnercustomgeneric}
}

\newcustomtheorem{assumption}{Assumption}

\newcommand{\IR}{\mathbb{R}}
\newcommand{\IC}{\mathbb{C}}
\newcommand{\IN}{\mathbb{N}}
\newcommand{\IZ}{\mathbb{Z}}

\newcommand{\IF}{\mathbb{F}}
\newcommand{\IE}{\mathbb{E}}

\newcommand{\IB}{\mathbb{B}}

\newcommand{\cA}{\mathcal{A}}

\newcommand{\cL}{\mathcal{L}}

\newcommand{\fu}{\mathfrak{u}}

\renewcommand{\L}{\mathrm{L}}
\newcommand{\C}{\mathrm{C}}
\newcommand{\B}{\mathrm{B}}

\newcommand{\W}{\mathrm{W}}

\newcommand{\per}{\mathrm{per}}

\newcommand{\e}{\mathrm{e}}

\renewcommand{\d}{\mathrm{d}}
\newcommand{\eps}{\varepsilon}

\renewcommand\Re{\operatorname{Re}}

\newcommand{\Lop}{\mathcal{L}}

\DeclareMathOperator{\Id}{Id}

\numberwithin{equation}{section} 

\title[Strong Time Periodic Solutions to the Bidomain Equations]{Strong Time Periodic Solutions to the Bidomain Equations with FitzHugh--Nagumo Type Nonlinearities}
\author{Matthias Hieber}
\author{Naoto Kajiwara}
\author{Klaus Kress}
\author{Patrick Tolksdorf}
\address{Fachbereich Mathematik, Technische Universit\"at Darmstadt, Schlossgartenstr. 7, 64289 Darmstadt, Germany}
\address{Graduate School of Mathematical Sciences, The University of Tokyo, 3-8-1 Komaba, Meguro, Tokyo, 153-8914, Japan}
\address{Fachbereich Mathematik, Technische Universit\"at Darmstadt, Schlossgartenstr. 7, 64289 Darmstadt, Germany}
\address{Fachbereich Mathematik, Technische Universit\"at Darmstadt, Schlossgartenstr. 7, 64289 Darmstadt, Germany}
\email{hieber@mathematik.tu-darmstadt.de}
\email{kajiwara@ms.u-tokyo.ac.jp}
\email{kkress@mathematik.tu-darmstadt.de}
\email{tolksdorf@mathematik.tu-darmstadt.de}
\thanks{The second author was supported by JSPS Japanese-German Graduate Externship and scientific grant Kiban S (26220702) as a research assistant of the University of Tokyo. 
The third and fourth authors are supported by the DFG International Research Training Group 1529 on Mathematical Fluid Dynamics at TU Darmstadt. }

\begin{document}

\begin{abstract}
Consider the bidomain equations subject to ionic transport described by the models of FitzHugh--Nagumo, Aliev--Panfilov, or Rogers--McCulloch. It is proved that this set of equations admits 
a unique, {\em strong} $T$-periodic solution provided it is innervated by $T$-periodic intra- and extracellular currents. The approach relies on a new periodic version of the classical 
Da Prato--Grisvard theorem on maximal $\L^p$-regularity in real interpolation spaces.   
\end{abstract}
\maketitle
\section{Introduction}\label{Sec: Introduction}


The bidomain system is a well established system of equations describing the electrical activities of the heart. For a detailed description of this model as well as 
its derivation from general principles, we refer, e.g., to~\cite{MR3308707,henriquez} and the monograph  by Keener and Sneyd ~\cite{MR1673204}. The system is given by 
\begin{align}\tag{BDE}\label{Eq: BDE}
\left\{
\begin{aligned}
\partial_t u + F(u,w) - \nabla \cdot (\sigma_i \nabla u_i) &= I_i & \mathrm{in} & ~ (0 , \infty) \times \Omega ,  \\
\partial_t u + F(u,w) + \nabla \cdot (\sigma_e \nabla u_e) &= - I_e & \mathrm{in} & ~ (0 , \infty) \times \Omega ,  \\
\partial_t w + G(u,w) &= 0 & \mathrm{in} & ~ (0 , \infty) \times \Omega ,  \\
\end{aligned}
\right.
\end{align}
subject to the boundary conditions 
\begin{align}
\sigma_i \nabla u_i \cdot \nu = 0, \quad  \sigma_e \nabla u_e \cdot \nu = 0  \quad \mathrm{on}  \quad  (0 , \infty) \times \partial \Omega , 
\end{align}
and the initial data 
\begin{align}
u(0) = u_0, \quad w(0) = w_0  \quad \mathrm{in}  \quad \Omega. 
\end{align}
Here $\Omega \subset \IR^n$ denotes a domain describing the myocardium,  the functions $u_i$ and $u_e$ model the  intra- and extracellular electric potentials, $u := u_i- u_e$  
denotes the transmembrane potential, and $\nu$ denotes the outward unit normal vector to $\partial \Omega$. The anisotropic properties of the intra- and extracellular tissue 
parts will be described by the conductivity matrices $\sigma_{i}(x)$ and $\sigma_{e}(x)$. Furthermore, $I_i$ and $I_e$ stand for the intra- and extracellular stimulation current, respectively. 

The variable $w$, the so-called gating variable, corresponds to the ionic transport through the cell membrane. On a microscopic level, the intra- and extracellular quantities are defined on 
disjoint domains $\Omega_i$ and $\Omega_e$ of $\Omega$. After a homogenization procedure described rather rigorously, e.g.,  in~\cite{PSC05, MR1944157}, one obtains the macroscopic model above, 
where the intra- and extracellular components are defined on all of $\Omega$. The behavior of the ionic current through the cell membrane, described by the variable $w$, is coupled with the transmembrane voltage $u$ by 
the equation in the third line of~\eqref{Eq: BDE}.      

Mathematical models describing the propagation of impulses in electrophysiology have a long tradition starting with the classical model by Hodgkin and Huxley in the 1950s, see, e.g., the recent survey article of Stevens~\cite{Stevens}. In this article, 
we consider various models for the ionic transport including the models by FitzHugh--Nagumo, Aliev--Panfilov, and Rogers--McCulloch. The {\em FitzHugh--Nagumo model} reads as 
\begin{align*}
F(u,w)&= u(u-a)(u-1)+w = u^3-(a+1)u^2+au+w , \\
G(u,w)&= bw - cu,
\end{align*}
where $0<a<1$ and $b$, $c > 0$. In the {\em Aliev--Panfilov model} the functions $F$ and $G$ are given by 
\begin{align*}
F(u,w)&= ku(u-a)(u-1)+uw = ku^3-k(a+1)u^2+kau+uw, \\
G(u,w)&=ku(u-1-a)+dw,
\end{align*}
whereas for the {\em Rogers--McCulloch model} we have 
\begin{align*}
F(u,w)&= bu(u-a)(u-1)+uw = bu^3-b(a+1)u^2+bau+uw, \\ 
G(u,w)&= dw - cu.
\end{align*} 
The coefficients in these models satisfy the conditions $0<a<1$ and $b$, $c$, $d$, $k > 0$.

Despite its importance in cardiac electrophysiology, not many analytical results on the bidomain equations are known until today. Note that the so-called bidomain operator is a very 
{\em non local} operator, which makes the analysis of this equation seriously more complicated compared, e.g., to the classical Allen--Cahn equation.   

The rigorous mathematical analysis of this system started with the work of Colli-Franzone and Savar\'e~\cite{MR1944157}, who introduced a variational formulation of the problem and showed 
the global existence and uniqueness of weak and strong solutions for FitzHugh-Nagumo model. Veneroni~\cite{MR2474265} extended the latter result to more general models for the ionic transport including the Luo and Rudy I model~\cite{LR91}. 

In 2009, a new approach to this system was presented by Bourgault, Cordi\`ere, and Pierre in~\cite{MR2451724}. They introduced for the first time the so-called bidomain operator within the $\L^2$-setting and showed that it is a non-negative and self-adjoint operator. 
By making use of  the theory of evolution equations they further showed the existence and uniqueness of a local strong solution and the existence of a global, weak solution to the system above for a large class of ionic models including the FitzHugh--Nagumo, Aliev--Panfilov, and Rogers--McCulloch models above.
In~\cite{KW13}, the uniqueness and regularity of the weak solution were proved. 

For results concerning the optimal control problem subject to the monodomain approximation, in which the conductivity matrices satisfy $\sigma_i = \lambda \sigma_e$ for some $\lambda >0$, we 
refer to a series of papers by Kunisch et al. ~\cite{MR3342160, KW12, NKP13, KNW}, see also \cite{SNMCLT06}. 
 
A new impetus to the field was recently given by  Giga and Kajiwara~\cite{GK}, who investigated the bidomain equations within the $\L^p$-setting for $1<p\leq \infty$. They showed that the 
bidomain operator is the generator of an analytic semigroup on $\L^p(\Omega)$ for $p \in (1,\infty]$ and constructed a local, strong solution to the bidomain system within this setting. 

All these results mainly concern the well-posedness of the bidomain equations and results on the dynamics of the solution are even more rare. We refer here to the very recent work of 
Mori and Matano~\cite{MM16}, who studied for the first time the stability of front solutions of the bidomain equations.  

In this context it is now a very natural question to ask, whether the bidomain equations admit time periodic solutions. Periodic solutions can be formulated in various regularity 
classes, ranging from weak over mild to strong solutions.  

In this paper, we consider the situation where the bidomain model, combined with one of the models for the ionic transport above, is innervated by  
{\em periodic intra- and extracellular currents}  $I_i$ and $I_e$. It is then our aim to show that in this case the innervated system admits a {\em strong} time periodic solution of 
period $T$ provided the outer forces $I_i$ and $I_e$ are both time-periodic of period $T>0$.
 
Let us emphasize, that we consider here the {\em full} bidomain model taking into account the anisotropic phenomena and not only the so-called monodomain approximation. A function space 
related to a  fixed point argument for the Poincar\'e map in the strong sense is naturally linked to a space of maximal regularity. This leads us to the scale of real interpolation spaces and  
our approach is then based on a periodic version of the classical Da Prato--Grisvard theorem~\cite{DPG75}. A different approach within the $\L^p$-setting based on a semilinear version of a  
result by Arendt and  Bu~\cite{AB02} on strong periodic solutions of linear equations would require additional properties of the bidomain operator, which, however, seem to be unknown. 
  
Some more specific words about the strategy of our approach are in order. The bidomain system is first reformulated into a coupled system. In this coupled system a $2 \times 2$ operator 
matrix $\cA$ involving the bidomain operator $A$ in one of its components will represent the linear part of ~\eqref{Eq: BDE}. Given a Banach space $X$ 
and a $T$-periodic function $f: \IR \to X$ whose restriction to $(0,T)$ belongs to $\L^p(0,T;X)$, we understand by a strong $T$-periodic solution to the bidomain equations with 
right-hand side $(f , 0)$ a $T$-periodic tupel $(u , w) \in \L^p(0,T;X)$ satisfying $(u^{\prime} , w^{\prime}) \in \L^p(0,T;X)$ and $ \cA (u , w) \in \L^p(0,T;X)$. This means in particular 
that $(u , w)$ admit the property of maximal $\L^p$-regularity. In order to obtain a $T$-periodic solution to \eqref{Eq: BDE} within this regularity class, we choose as underlying 
Banach space the real interpolation space $D_A(\theta , p)$ for $\theta \in (0,1)$, $1\leq p<\infty$, and $A$ being again the bidomain operator. Our approach to $T$-periodic solutions 
for the linearized equation is then  based  on a periodic version of the classical Da Prato--Grisvard theorem, which we develop in Section~\ref{Sec: Da Prato--Grisvard}. Having this 
at hand, we apply then the contraction mapping principle in the space of maximal regularity to find  a 
strong $T$-periodic solution of the nonlinear problem in a neighborhood of stable equilibrium points.

This paper is organized as follows: While Section~\ref{Sec: Preliminaries} is devoted to fix some notation and to collect some known results, our main results on strong $T$-periodic solutions to the bidomain equations subject to a large class of models for the ionic transport are presented in Section~\ref{Sec: Main results}.
The following Section~\ref{Sec: Da Prato--Grisvard} presents a  periodic version of the Da Prato--Grisvard theorem, which will be extended in Section~\ref{Sec: The periodic solution} to the semilinear setting. In Section~\ref{Sec: Examples} we apply 
our previous results to the bidomain equations subject to  various models for the ionic transport including the models by FitzHugh--Nagumo, Aliev--Panfilov, and Rogers--McCulloch.

\section{Preliminaries} 
\label{Sec: Preliminaries}
In the whole article, let the space dimension $n \geq 2$ be fixed and let $\Omega \subset \IR^n$ denote a bounded domain with boundary $\partial \Omega$ of class $\C^2$. For the conductivity matrices $\sigma_i$ and $\sigma_e$ we make the following assumptions.

\begin{assumption}{E}
\label{Ass: Coefficients}
The conductivity matrices $\sigma_i , \sigma_e : \overline{\Omega} \to \IR^{n \times n}$ are symmetric matrices and are functions of class $\C^1 (\overline{\Omega})$. Ellipticity is imposed by means of the following condition: there exist constants 
$\underline{\sigma}$, $\overline{\sigma}$ with $0 < \underline{\sigma} < \overline{\sigma}$ such that 
\begin{align}
\underline{\sigma} |\xi|^2 \leq \langle \sigma_i (x) \xi , \xi \rangle \leq \overline{\sigma}|\xi|^2 \quad \text{and} \quad \underline{\sigma} |\xi|^2 \leq \langle \sigma_e(x) \xi , \xi \rangle \leq \overline{\sigma}|\xi|^2 \label{Eq: UE}
\end{align}
for all $x \in \overline{\Omega}$ and all $\xi \in \IR^n$. 
Moreover, it is assumed that 
\begin{align} \begin{aligned}
\sigma_i \nabla u_i \cdot \nu = 0 \qquad &\Leftrightarrow \qquad \nabla u_i \cdot \nu =0 \quad &{\rm on}~\partial\Omega, \\
 \sigma_e \nabla u_e \cdot \nu = 0 \qquad &\Leftrightarrow \qquad \nabla u_e \cdot \nu =0 \quad &{\rm on}~\partial\Omega. \label{Eq: EV}
\end{aligned}
\end{align}
\end{assumption}
It is known due to~\cite{franzone1990mathematical} that~\eqref{Eq: EV} is a biological reasonable assumption.

Next, we define the bidomain operator in the $\L^q$-setting for $1 < q < \infty$. To this end, let $\L^q_{av} (\Omega):= \{u\in \L^q(\Omega) : \int_\Omega u \; \d x = 0\}$ and 
let $P_{av}$ be the orthogonal projection from $\L^q(\Omega)$ to $\L^q_{av}(\Omega)$, i.e., $P_{av} u := u - \frac{1}{|\Omega|} \int_\Omega u \; \d x$. 
We then introduce the  operators $A_i$ and $A_e$ by   
\begin{align*}
A_{i,e} u  &:= - \nabla \cdot (\sigma_{i,e} \nabla u), \\
D(A_{i,e}) &:= \left\{ u \in \W^{2,q} (\Omega) \cap \L^q_{av}(\Omega) : \sigma_{i,e} \nabla u\cdot \nu =0 {\rm~a.e.~on~}\partial \Omega \right\} \subset \L^q_{av}(\Omega),
\end{align*}
where $A_{i , e}$ and $\sigma_{i , e}$ indicates that either $A_i$ and $\sigma_i$ or $A_e$ and $\sigma_e$ are considered.
Due to condition~\eqref{Eq: EV} we obtain  $D (A_i) = D (A_e)$ and thus, it is possible to define the sum  $A_i + A_e$ of $A_i$ and $A_e$ with the domain $D(A_i)= D(A_e)$.
Note that the inverse operator $(A_i + A_e)^{-1}$ on $\L^q_{av} (\Omega)$ is a bounded linear operator. 

Following~\cite{GK} we define the bidomain operator as follows. Let $\sigma_i$ and $\sigma_e$ satisfy Assumption~\ref{Ass: Coefficients}. Then the bidomain operator $A$ is defined as 
\begin{align}
A = A_i (A_i+A_e)^{-1} A_e P_{av}
\end{align}
with domain 
\begin{align*}
D(A) := \{u \in \W^{2,q}(\Omega) : \nabla u \cdot \nu = 0 {\rm~a.e.~on~} \partial \Omega\}.
\end{align*}

The following resolvent estimates for $A$ were proven by Giga and Kajiwara  in~\cite{GK}. Here, denote for $\theta \in (0,\pi]$ the sector $\Sigma_{\theta}:=\{\lambda \in \IC \setminus \{0\}: |\mathrm{arg} \lambda | < \theta\}$.

\begin{proposition}[{\cite[Theorem 4.7, Theorem 4.9]{GK}}]\label{Thm: EU}
\label{Prop: Giga-Kajiwara}
Let $1 < q < \infty$, $\Omega$ be a bounded $\C^2$-domain and let $\sigma_i$ and $\sigma_e$ satisfy Assumption~\ref{Ass: Coefficients}. 
Then, for $\lambda \in \Sigma_{\pi}$ and $f \in \L^q(\Omega)$, the resolvent problem 
\begin{align}
(\lambda + A) u = f \quad {\rm in} \ \Omega
\end{align}
has a unique solution $u \in D(A)$. Moreover, for each $\varepsilon \in (0 , \pi / 2)$ there exists a constant $C > 0$ such that for all $\lambda \in \Sigma_{\pi - \varepsilon}$ and all 
$ f \in \L^q(\Omega)$ the unique solution $u \in D(A)$ satisfies 
\begin{align*}
|\lambda|  \| u \|_{\L^q (\Omega)}+|\lambda|^{1/2} \| \nabla u\|_{\L^q(\Omega)}+\| \nabla^2 u \|_{\L^q(\Omega)} \leq C \| f \|_{\L^q(\Omega)}.
\end{align*}
\end{proposition}

Observe that the proposition above implies in particular that $- A$ generates a bounded analytic semigroup $\e^{-tA}$ on $\L^q(\Omega)$. 

Under the assumption of the conservation of currents, i.e., 
\begin{align}
\int_\Omega(I_i(t)+I_e(t))\; \d x=0, \quad t\geq 0 \label{Eq: CC}
\end{align}
and assuming moreover $\int_\Omega u_e \; \d x=0$, the bidomain equations~\eqref{Eq: BDE} may be equivalently rewritten as an evolution equation~\cite{MR2451724,GK} of the form 
\begin{align*}
\tag{ABDE}\label{Eq: ABDE}
\left\{
\begin{aligned}
\partial_{t}u + Au + F(u,w)&=I, \qquad &\mathrm{in} \ (0,\infty),\\
\partial_{t}w + G(u,w)&=0, &\mathrm{in} \ (0,\infty),\\
u (0)&= u_0,\\
w(0)&= w_0,
\end{aligned}
\right.
\end{align*}
where 
\begin{align}
\label{Eq: Modified source term}
 I:=I_i-A_i(A_i+A_e)^{-1}(I_i+I_e)
\end{align}
is the modified source term. The functions $u_e$ and $u_i$ can be recovered from $u$ by virtue of the following relations
\begin{align*}
u_e&=(A_i+A_e)^{-1}\{(I_i+I_e)-A_iP_{av}u\},\\
u_i&=u+u_e.
\end{align*}
Our main results on the unique existence of strong $T$-periodic solutions to~\eqref{Eq: ABDE} are formulated in the real interpolation space $D_A(\theta,p)$ between $D(A)$ and the 
underlying space $\L^q(\Omega)$. This choice of spaces is motivated by our aim to prove the existence and uniqueness of $T$-periodic solutions to the bidomain equations in the 
{\em strong}, and not only in the  mild sense. The classical Da Prato--Grisvard theorem ensures the maximal $\L^p$-regularity for parabolic evolution equations in these spaces and our approach 
is based on a {\em periodic version of the Da Prato--Grisvard theorem}.

More specifically, let $X$ be a Banach space and $-\cA$ be the generator of a bounded analytic semigroup $\e^{-t \cA}$ on $X$ with domain $D(\cA)$. 
For $\theta \in (0 , 1)$ and $1 \leq p < \infty$, we denote by $D_{\cA}(\theta,p)$ space defined as  
\begin{align}\label{Eq: Characterization real interpolation space}
D_{\cA}(\theta , p) := \Big\{ x \in X : [x]_{\theta,p}:= \Big(\int_0^\infty\|t^{1 - \theta} \cA \e ^{- t \cA} x\|_X^p \; \frac{\d t}{t}\Big)^{1/p} < \infty \Big\}.
\end{align}
When equipped with the norm $\|x\|_{\theta,p}:=\|x\| + [x]_{\theta,p}$, the space $D_{\cA}(\theta,p)$ becomes a Banach space. For details and more on interpolation spaces we refer, e.g., 
to~\cite{Lun95, Lun09}. It is well-known that $D_{\cA} (\theta , p)$ coincides with the real interpolation space $(X , D(\cA))_{\theta , p}$ and that the respective norms are equivalent. If 
$0 \in \rho (\cA)$, then the real interpolation space norm is equivalent to the homogeneous norm $[\cdot]_{\theta,p}$, see~\cite[Corollary 6.5.5]{haase2006}. Consider in particular the 
bidomain operator $A$ in $X=L^q(\Omega)$ for $1 < q < \infty$. Then, following Amann~\cite[Theorem 5.2]{Ama93}, the space $(X , D(A))_{\theta , p}$ can be characterized as 
\begin{align}\label{besov}
(\L^q(\Omega),D(A))_{\theta, p} = \B^{2\theta}_{q,p}(\Omega), \quad 1\leq p \leq \infty,
\end{align}
provided $2\theta  \in (0,1 + 1/q)$. Here $\B^s_{q,p}(\Omega)$ denotes, as usual, the Besov space of order $s \geq 0$.

For $0 < T < \infty$, we define the solution space $\IE_{\cA}^{\per}$ as  
\begin{align*}
  \IE_{\cA}^{\per} := \{ u \in \W^{1 , p} (0 , T ; D_{\cA}(\theta , p)) : \cA u \in \L^p(0 , T ; D_{\cA}(\theta , p)) \text{ and } u(0) = u(T) \}
\end{align*}
with norm
\begin{align*}
  \| u \|_{\IE_{\cA}^{\per}} := \| u \|_{\W^{1 , p} (0 , T ; D_{\cA}(\theta , p))} + \| \cA u \|_{\L^p(0 , T ; D_{\cA} (\theta , p))}, 
\end{align*}
which corresponds to the data space
\begin{align*}
 \IF_{\cA} := \L^p(0,T;D_{\cA}(\theta,p)).
\end{align*}
  
In our situation, where $A$ denotes the bidomain operator, the solution space for the transmembrane potential $u$ reads as
\begin{align*}
  \IE_A^{\per} = \{ u \in \W^{1 , p} (0 , T ; D_A(\theta , p)) : A u \in \L^p(0 , T ; D_A(\theta , p)) \text{ and } u(0) = u(T) \}.
\end{align*}
The solution space for the gating variable $w$ is defined as
\begin{align*}
 \IE_w^{\per} := \{ w \in \W^{1 , p}(0 , T ; D_A(\theta , p)) : w(0) = w(T) \}.
\end{align*}
Then, the solution space for the periodic bidomain system is defined as the product space
\begin{align*}
\IE := \IE_A^{\per} \times \IE_w^{\per}.
\end{align*}

Finally, for a Banach space $X$ we denote by $\IB^X(u^*,R)$ the closed ball in $X$ with center $u^*\in X$ and radius $R>0$, i.e., 
\begin{align*}
\IB^X(u^*,R):=\{u \in X: \|u-u^*\|_X \leq R\}.
\end{align*}

\section{Main results for various models}\label{Sec: Main results}

In this section we state our main results concerning the existence and uniqueness of strong $T$-periodic solutions to the bidomain equations subject to various models of the ionic 
transport. Notice that the respective models treated here are slightly more general as described in the introduction, as an additional parameter $\eps > 0$ is introduced, that incorporates the phenomenon of fast and slow diffusion. \par
Additionally to Assumption~\ref{Ass: Coefficients} on the conductivity matrices of the bidomain operator $A$, we require the following regularity and periodicity conditions on the forcing term $I$.

\begin{assumption}{P}
\label{Ass: Forcing}
Let $1 \leq p < \infty$ and $n < q < \infty$ satisfy $1/p + n/(2q) \leq 3/4$. 
Assume $I: \IR \to D_A(\theta,p)$ is a $T$-periodic function satisfying  $I_{|(0,T)} \in \IF_A$ for some $\theta \in (0,1/2)$ and $T > 0$.
\end{assumption}

\begin{remark}
If $\Omega$ has a $\C^4$-boundary and if the conductivity matrices $\sigma_i$ and $\sigma_e$ lie in $\W^{3 , \infty} (\Omega ; \IR^{n \times n})$, then Assumption~\ref{Ass: Forcing} is satisfied by virtue of~\eqref{Eq: Modified source term} if $I_i , I_e: \IR \to D_A(\theta,p)$ are $T$-periodic functions satisfying $I_{i|_{(0,T)}}$ and $I_{e|_{(0,T)}} \in \IF_A$. Indeed, this follows by real interpolation since $A_i (A_i + A_e)^{-1}$ is bounded on $\L^q_{av} (\Omega)$ and from $D(A)\cap \L^q_{av} (\Omega)$ into $\W^{2 , q} (\Omega)\cap \L^q_{av} (\Omega)$.
\end{remark}

We start with the most classical model due to FitzHugh and Nagumo.

\subsection{The periodic bidomain FitzHugh--Nagumo model}\label{Subsec: FitzHugh--Nagumo}\mbox{}\\

For $T > 0$, $0 < a < 1$, and $b , c , \eps > 0$, the periodic bidomain FitzHugh--Nagumo equations are given by
\begin{align}
\left\{
\begin{aligned}
\partial_t u+\varepsilon Au&=I - \frac{1}{\varepsilon}[u^3-(a+1)u^2+au+w] \qquad&{\rm in}\ \IR\times \Omega,\\
\partial_t w &=cu-bw &{\rm in}\ \IR\times \Omega,\\
u(t) &= u(t+T) &{\rm in} \ \IR \times \Omega,\\
w(t) &= w(t+T) &{\rm in} \ \IR \times \Omega.
\end{aligned}\label{Eq: bfhn}
\right.
\end{align}
This system has three  equilibrium points, the trivial one  $(u_1,w_1) = (0,0)$ and two others given by  $(u_2,w_2)$ and $(u_3,w_3)$, where
\begin{align}
u_2=\frac12(a+1-d), \quad w_2= \frac{c}{2b}(a+1-d), \quad   u_3=\frac12(a+1 + d), \quad &w_3= \frac{c}{2b} (a+1+d), 
\end{align}
and $d=\sqrt{(a+1)^2-4(a+\frac{c}{b})}$.  We assume that the following stability condition $\mathrm{(S_{FN})}$ on the coefficients is satisfied:
\begin{align} \tag{S\textsubscript{FN}} \label{Eq: SFN}
 c< b \left(\frac{(a-1)^2}{4} - a\right)\quad \text{and} \quad u_3> \frac13\left(a+1+\sqrt{(a+1)^2-3a}\right).
\end{align}

Our result on strong periodic solutions to the bidomain FitzHugh--Nagumo equations reads then as follows.

\begin{theorem}\label{Thm: FitzHugh-Nagumo}
Let $\Omega \subset \IR^n$, $n \geq 2$, be a bounded $\C^2$-domain and suppose that Assumptions~\ref{Ass: Coefficients} and~\ref{Ass: Forcing} hold true.
\begin{enumerate}
\item[a)] Then there exist constants $R>0$ and $C(R)>0$ such that if $\|I\|_{\IF_A} <C(R)$, the equation~\eqref{Eq: bfhn} admits a unique T-periodic strong 
solution $(u,w)$ with $(u , w)_{|(0 , T)} \in \IB^{\IE}((0,0),R)$. \\
\item[b)] If condition~\eqref{Eq: SFN} is satisfied, then there exist constants $R>0$ and $C(R)>0$ such that if $\|I\|_{\IF_A}<C(R)$, the equation~\eqref{Eq: bfhn} admits a unique T-periodic strong solution $(u,w)$ with $(u,w)_{|(0 , T)} \in \IB^{\IE}((u_3,w_3),R)$. 
\end{enumerate}
\end{theorem}

\subsection{The periodic bidomain Aliev--Panfilov model}\label{Subsec: Aliev--Panfilov}\mbox{}\\

For $T > 0$, $0 < a < 1$, and $d , k , \eps > 0$, the periodic bidomain Aliev--Panfilov equations are given by
\begin{align}
\left\{
\begin{aligned}
\partial_t u+ \varepsilon Au&=I-\frac{1}{\varepsilon}[ku^3-k(a+1)u^2+kau+uw] \qquad &{\rm in}\ \IR\times \Omega,\\
\partial_t w &= -(ku(u-1-a)+dw) &{\rm in}\ \IR\times \Omega,\\
u(t) &= u(t+T) &{\rm in} \ \IR \times \Omega,\\
w(t) &= w(t+T) &{\rm in} \ \IR \times \Omega.
\end{aligned} \label{Eq: bap}
\right.
\end{align}

This system has only one stable equilibrium point, namely the trivial solution $(u_1,w_1) = (0,0)$. Our theorem on the existence and uniqueness of strong, periodic solutions to the 
periodic bidomain Aliev--Panfilov equations reads as follows.

\begin{theorem}\label{Thm: Aliev-Panfilov}
Let $\Omega \subset \IR^n$, $n \geq 2$, be a bounded $\C^2$-domain and suppose that Assumptions~\ref{Ass: Coefficients} and~\ref{Ass: Forcing} hold true. Then, there exist constants $R>0$ and $C(R)>0$ such that if $\|I\|_{\IF_A} <C(R)$, the equation~\eqref{Eq: bap} admits a unique T-periodic strong solution $(u , w)$ with $(u , w)_{|(0 , T)} \in \IB^{\IE}((0,0),R)$. 
\end{theorem}

\subsection{The periodic bidomain Rogers--McCulloch model}\label{Subsec: Rogers--McCulloch}\mbox{}\\

For $T > 0$, $0 < a < 1$, and $b , c , d , \eps > 0$, the periodic bidomain Rogers--McCulloch equations are given by
\begin{align}
\left\{
\begin{aligned} 
\partial_t u+\varepsilon Au&=I-\frac{1}{\varepsilon}[bu^3-b(a+1)u^2+bau+uw] \qquad &{\rm in}\ \IR\times \Omega,\\
\partial_t w &= cu-dw \qquad &{\rm in}\ \IR\times \Omega,\\
u(t) &= u(t+T) &{\rm in} \ \IR \times \Omega,\\
w(t) &= w(t+T) &{\rm in} \ \IR \times \Omega.
\end{aligned} \label{Eq: bmc}
\right.
\end{align}

This system has three  equilibrium points, the trivial one  $(u_1,w_1) = (0,0)$ and two others given by  $(u_2,w_2)$ and $(u_3,w_3)$, where
\begin{align}
&u_2=\frac12(a+1-\frac{c}{bd}-e), \quad w_2= \frac{c}{2d}(a+1-\frac{c}{bd}-e),\\  &u_3=\frac12(a+1 -\frac{c}{bd}+ e), \quad w_3= \frac{c}{2d} (a+1-\frac{c}{bd}+e), 
\end{align}
and $e=\sqrt{\left(a+1-\frac{c}{bd}\right)^2-4a}$.  We assume that the following stability condition~\eqref{Eq: SRM} on the coefficients is satisfied:
\begin{align} \tag{S\textsubscript{RM}} \label{Eq: SRM}
 \sqrt{\big(a+1-\frac{c}{bd}\big)^2-4a}-\frac{c}{bd}>0.
\end{align}

Our theorem on the existence and uniqueness of strong periodic solutions to the periodic bidomain Rogers--McCulloch equations reads as follows.

\begin{theorem}
\label{Thm: Rogers-McCulloch}
Let $\Omega \subset \IR^n$, $n \geq 2$, be a bounded $\C^2$-domain and suppose that Assumptions~\ref{Ass: Coefficients} and~\ref{Ass: Forcing} hold true.
\begin{enumerate}
\item[a)] Then there exist constants $R>0$ and $C(R)>0$ such that if $\|I\|_{\IF_A} < C(R)$, the equation~\eqref{Eq: bmc} admits a unique T-periodic strong 
solution $(u,w)$ with $(u , w)_{|(0 , T)} \in \IB^{\IE}((0,0),R)$. \\
\item[b)] If condition~\eqref{Eq: SRM} is satisfied, then there exist constants $R>0$ and $C(R)>0$ such that if $\|I\|_{\IF_A}<C(R)$, the equation~\eqref{Eq: bmc} admits a unique T-periodic strong solution $(u,w)$ with $(u,w)_{|(0 , T)} \in \IB^{\IE}((u_3,w_3),R)$. 
\end{enumerate}
\end{theorem}

\subsection{The periodic bidomain Allen--Cahn equation}\label{Subsec: Allen--Cahn}\mbox{}\\

For $T > 0$, the periodic bidomain Allen--Cahn equation is given by
\begin{align}
\left\{
\begin{aligned}
\partial_t u +Au&=I+u-u^3 \qquad &{\rm in}\ \IR\times \Omega,\\
u(t) &= u(t+T) &{\rm in} \ \IR \times \Omega.
\end{aligned} \label{Eq: bac}
\right.
\end{align} 
This system has three equilibrium points, $u_1 = - 1$, $u_2 = 0$, and $u_3 = 1$ and our  theorem on the existence and uniqueness of strong, periodic solutions to the periodic bidomain 
Allen--Cahn equation reads as follows.

\begin{theorem}\label{Thm: Allen-Cahn}
Let $\Omega \subset \IR^n$, $n \geq 2$, be a bounded $\C^2$-domain and suppose that Assumptions~\ref{Ass: Coefficients} and~\ref{Ass: Forcing} hold true.
\begin{enumerate}
 \item[a)] Then, there exist constants $R>0$ and $C(R)>0$ such that if $\|I\|_{\IF_A} < C(R)$ the equation~\eqref{Eq: bac} admits a unique T-periodic strong solutions $u$ with $u_{|(0 , T)} \in \IB^{\IE_A^{\per}}(-1,R)$.
 \item[b)] Then, there exist constants $R>0$ and $C(R)>0$ such that if $\|I\|_{\IF_A} < C(R)$ the equation~\eqref{Eq: bac} admits a unique T-periodic strong solutions $u$ with $u_{|(0 , T)} \in \IB^{\IE_A^{\per}}(1,R)$.
\end{enumerate}
\end{theorem}

\section{A periodic version of the Da Prato--Grisvard theorem}
\label{Sec: Da Prato--Grisvard}

Let $X$ be a Banach space and $- \cA$ be the generator of a bounded analytic semigroup on $X$. Assume that  $\theta \in (0 , 1)$, $1 \leq p < \infty$, and $0 < T < \infty$. Then, for 
$f \in \L^p(0 , T ; D_{\cA}(\theta , p))$ we consider
\begin{align}
\label{Eq: Variation of constants}
 u(t) := \int_0^t \e^{- (t - s) \cA} f(s) \; \d s, \qquad 0 < t < T.
\end{align}
Then, $u$ is the unique mild solution to the abstract Cauchy problem
\begin{align} \tag{ACP} \label{Eq: ACP}
 \left\{ \begin{aligned}
  u^{\prime} (t) + \cA u (t) &= f(t), \qquad 0 < t < T \\
  u(0) &= 0
 \end{aligned} \right.
\end{align}
and fulfills, thanks to the classical Da Prato and Grisvard theorem~\cite{DPG75}, the following maximal regularity estimate.

\begin{proposition}[{\cite[Da Prato, Grisvard]{DPG75}}]\label{Thm: Da Prato-Grisvard}
Let $\theta \in (0 , 1)$, $1 \leq p < \infty$, and $0 < T < \infty$. Then there exists a constant $C > 0$ such that for all $f \in \L^p(0 , T ; D_{\cA}(\theta , p))$, the function $u$ given 
by~\eqref{Eq: Variation of constants} satisfies $u(t) \in D(\cA)$ for almost every $0 < t < T$ and
\begin{align*}
 \| \cA u \|_{\L^p(0 , T ; D_{\cA}(\theta , p))} \leq C \| f \|_{\L^p(0 , T ; D_{\cA}(\theta , p))}.
\end{align*}
\end{proposition}

We remark at this point that the theorem above implies that the mild solution $u$ to $\mathrm{(ACP)}$ is in fact a strong solution satisfying $u^{\prime} (t) + \cA u (t) = f(t)$ for almost 
every $0 < t < T$.

The proof of our main results are based on the following periodic version of the Da Prato--Grisvard theorem, which is also of independent interest. To this end, we define 
the periodicity of measurable functions as follows. 
For some $0 < T < \infty$, we say a measurable function $f : \IR \to X$ is called 
\textit{periodic of period $T$} if $f (t) = f(t + T)$ holds true for almost all  $t \in (- \infty , \infty)$.

For $\theta \in (0 , 1)$, $1 \leq p < \infty$, and $0 < T < \infty$ assume that $f : \IR \to D_{\cA} (\theta , p)$ is periodic of period $T$. Then the periodic version of~\eqref{Eq: ACP} reads as
\begin{align} \tag{PACP} \label{Eq: PACP}
 \left\{ \begin{aligned}
  u^{\prime} (t) + \cA u(t) &= f(t), &&t \in \IR, \\
  u(t) &= u(t + T), \qquad &&t \in \IR.
 \end{aligned} \right.
\end{align}
Formally, a candidate for a solution $u$ of~\eqref{Eq: PACP} is given by
\begin{align}
\label{Eq: Definition of periodic u}
 u(t) := \int_{- \infty}^t \e^{- (t - s) \cA} f(s) \; \d s.
\end{align}
The following lemma shows that, under certain assumptions on $\cA$ and $f$, $u$ is indeed  well-defined, continuous and periodic.

\begin{lemma}\label{Lem: Properties of the mild solution}
Let  $f : \IR \to D_{\cA} (\theta,p)$ be a $T$-periodic function satisfying $f_{|(0 , T)} \in \L^p(0 , T ; D_{\cA} (\theta , p))$ and assume that $0 \in \rho(\cA)$. 
Then, the function $u$ defined by~\eqref{Eq: Definition of periodic u} is well-defined, satisfies $u \in \C( \IR ; D_{\cA} (\theta , p))$, and is $T$-periodic.
\end{lemma}

\begin{proof}
Let $k_0 \in \IZ$ be such that $- k_0 T < t \leq - (k_0 - 1) T$. Using H\"older's inequality, the periodicity of $f$, and the exponential decay of $\e^{-t \cA}$, we obtain
\begin{align*}
 \int_{-\infty}^t &\|\e^{- (t - s) \cA} f(s)\|_{D_{\cA}(\theta , p)} \; \d s \\
 &= \int_{- k_0 T}^t \|\e^{- (t - s) \cA} f(s)\|_{D_{\cA} (\theta , p)} \; \d s + \sum_{k = k_0}^{\infty} 
\int_{- (k + 1) T}^{- k T} \| \e^{- (t - s) \cA} f(s) \|_{D_{\cA} (\theta , p)} \; \d s \\
&\leq C \bigg( \int_0^{t + k_0 T} \| f(s) \|_{D_{\cA} (\theta , p)}^p \; \d s \bigg)^{\frac{1}{p}} + C \sum_{k = k_0}^{\infty} \e^{- \omega k T} \int_0^T \| \e^{- (T - s) \cA} f(s) \|_{D_{\cA} (\theta , p)} \; \d s\\
&\leq C \Big( 1 + \sum_{k = k_0}^{\infty} \e^{- \omega k T} \Big) \bigg( \int_0^T \| f(s) \|_{D_{\cA} (\theta , p)}^p \; \d s \bigg)^{\frac{1}{p}}
\end{align*}
for some $\omega>0$. It follows that $u$ is well-defined. For the continuity of $u$ we write for $h > 0$
\begin{align*}
 u(t + h) - u(t) = \int_t^{t + h} \e^{- (t + h - s) \cA} f(s) \; \d s + \int_{- \infty}^t \e^{- (t - s) \cA} [ \e^{- h \cA} - \Id ] f(s) \; \d s.
\end{align*}
By the boundedness of the semigroup it suffices to consider the second integral. This resembles the expression from the first part of the proof but with $f$ being replaced by 
$[ \e^{- h \cA} - \Id ] f$. Thus,
\begin{align*}
 \Big\| \int_{- \infty}^t \e^{- (t - s) \cA} [ \e^{- h \cA} - \Id ] f(s) \; \d s \Big\|_{D_{\cA} (\theta , p)} \leq C \bigg( \int_0^T \| [ \e^{- h \cA} - \Id ] f(s) \|_{D_{\cA} (\theta , p)}^p \; \d s 
\bigg)^{\frac{1}{p}}
\end{align*}
and the right-hand side tends to zero as $h \to 0$ by Lebesgue's theorem. 
The periodicity of $u$ directly follows by using the transformation $s^{\prime} = s + T$ and the periodicity of $f$.
\end{proof}

We now state  the periodic version of the Da Prato--Grisvard theorem.

\begin{theorem}\label{Thm: Da Prato-Grisvard periodic}
Let $X$ be a Banach space and $- \cA$ be the generator of a bounded analytic semigroup on $X$ with $0 \in \rho(\cA)$. Let  $\theta \in (0 , 1)$, $1 \leq p < \infty$, and $0 < T < \infty$. \par
Then there exists a constant $C > 0$ such that for all periodic functions $f : \IR \to D_{\cA} (\theta , p)$ with $f_{|(0 , T)} \in \L^p(0 , T ; D_{\cA} (\theta , p))$ the function $u$ 
defined by~\eqref{Eq: Definition of periodic u} lies in $\C(\IR ; D_{\cA} (\theta , p))$, is periodic of period $T$, satisfies $u (t) \in D(\cA)$ for almost every $t \in \IR$, and satisfies
\begin{align*}
 \| \cA u \|_{\L^p(0 , T ; D_{\cA} (\theta , p))} \leq C \| f \|_{\L^p(0 , T ; D_{\cA} (\theta , p))}.
\end{align*}
\end{theorem}

\begin{proof}
The continuity and periodicity of $u$ are proven in Lemma~\ref{Lem: Properties of the mild solution}. Let $t \in [0 , T)$ and use the transformation $s^{\prime} = s + (k + 1) T$ for $k \in \IN_0$ as well as that $f$ is periodic to write
\begin{align}
\label{Eq: u as a sum}
\begin{aligned}
 u(t) = \int_0^t \e^{- (t - s) \cA} f(s) \; \d s + \sum_{k = 0}^{\infty} \e^{- (t + k T) \cA} \int_0^T \e^{- (T - s) \cA} f(s) \; \d s.
\end{aligned}
\end{align}
In the following, use the notation
\begin{align*}
 \fu := \int_0^T \e^{- (T - s) \cA} f(s) \; \d s.
\end{align*}
Since Proposition~\ref{Thm: Da Prato-Grisvard} implies
\begin{align*}
 \int_0^t \e^{- (t - s) \cA} f(s) \; \d s \in D(\cA) \qquad (\text{a.e. } t \in (0 , T))
\end{align*}
and
\begin{align*}
 \Big\| t \mapsto \cA \int_0^t \e^{- (t - s) \cA} f(s) \; \d s \Big\|_{\L^p(0 , T ; D_{\cA} (\theta , p))} \leq C \| f \|_{\L^p(0 , T ; D_{\cA} (\theta , p))},
\end{align*}
by the exponential decay of the semigroup, it suffices to prove the estimate
\begin{align*}
 \| t \mapsto \cA \e^{- t \cA} \fu \|_{\L^p (0 , T ; D_{\cA} (\theta , p))} \leq C \| f \|_{\L^p(0 , T ; D_{\cA} (\theta , p))}.
\end{align*}
\begin{flushleft}
\textbf{Step 1.}
\end{flushleft}
Let $\gamma_1 , \gamma_2 \in (0 , 1)$ with $\gamma_1 + \gamma_2 = 1$ and $1 / p^{\prime} < \gamma_2 < 1 - \theta + 1 / p^{\prime}$, where $p^{\prime}$ denotes the H\"older conjugate exponent to $p$. Then, the boundedness and the analyticity of the semigroup, followed by a linear transformation and H\"older's inequality imply
\begin{align*}
 \| \cA \e^{- \tau \cA} \cA \e^{- t \cA} \fu \|_X &\leq C \int_0^T \frac{1}{(T + \tau + t - s)^{\gamma_1}} \frac{1}{(T + \tau + t - s)^{\gamma_2}} \| \cA \e^{- (T + \tau + t - s) / 2 \cA} f(s) \|_X \; \d s \\
 &= C \int_t^{T + t} \frac{1}{(\tau + s)^{\gamma_1}} \frac{1}{(\tau + s)^{\gamma_2}} \| \cA \e^{- (\tau + s) / 2 \cA} f(T + t - s) \|_X \; \d s \\
 &\leq C (\tau + t)^{1 / p^{\prime} - \gamma_2} \bigg( \int_t^{T + t} \frac{1}{(\tau + s)^{\gamma_1 p}} \| \cA \e^{- (\tau + s) / 2 \cA} f(T + t - s) \|_X^p \; \d s \bigg)^{\frac{1}{p}}.
\end{align*}
Notice that $1 / p^{\prime} < \gamma_2$ was eminent in the calculation above. Next, $t > 0$ implies
\begin{align}
\label{Eq: Estimate of Step 1 periodic case}
 \| \cA \e^{- \tau \cA} \cA \e^{- t \cA} \fu \|_X \leq C \tau^{1 / p^{\prime} - \gamma_2} \bigg( \int_t^{T + t} \frac{1}{(\tau + s)^{\gamma_1 p}} \| \cA \e^{- (\tau + s) / 2 \cA} f(T + t - s) \|_X^p \; \d s \bigg)^{\frac{1}{p}}.
\end{align}
\begin{flushleft}
\textbf{Step 2.}
\end{flushleft}
An application of~\eqref{Eq: Estimate of Step 1 periodic case} and Fubini's theorem yields
\begin{align*}
 \int_0^T \| \cA \e^{- \tau \cA} &\cA \e^{- t \cA} \fu \|_X^p \; \d t \\
 &\leq C \tau^{p ( 1 / p^{\prime} - \gamma_2)} \int_0^{2T} \int_{\max\{ 0 , s - T \}}^{\min\{T , s\}} \frac{1}{(\tau + s)^{\gamma_1 p}} \| \cA \e^{- (\tau + s) / 2 \cA} f(T + t - s) \|_X^p \; \d t \; \d s.
\end{align*}
Notice that the inner integral can be estimated by using $\min\{ T , s \} \leq s$. The transformation $t^{\prime} = T + t - s$ delivers then the estimate
\begin{align}
\label{Eq: Estimate of Step 2 periodic case}
\begin{aligned}
 \| t \mapsto \cA \e^{- \tau \cA} \cA \e^{- t \cA}& \fu \|_{\L^p(0 , T ; X)}^p \\
 &\leq C \tau^{p ( 1 / p^{\prime} - \gamma_2)} \int_0^{2T} \int_{\max\{ 0 , T - s \}}^T \frac{1}{(\tau + s)^{\gamma_1 p}} \| \cA \e^{- (\tau + s) / 2 \cA} f(t) \|_X^p \; \d t \; \d s.
\end{aligned}
\end{align}
\begin{flushleft}
\textbf{Step 3.}
\end{flushleft}
Use Fubini's theorem first and then~\eqref{Eq: Estimate of Step 2 periodic case} to estimate the full norm by
\begin{align*}
 \int_0^T [\cA \e^{- t \cA} &\fu]_{\theta , p}^p \; \d t 
 \leq C \int_0^{\infty} \tau^{\gamma - 1} \int_0^{2T} \int_{\max\{ 0 , T - s \}}^T \frac{1}{(\tau + s)^{\gamma_1 p}} \| \cA \e^{- (\tau + s) / 2 \cA} f(t) \|_X^p \; \d t \; \d s \; \d \tau,
\end{align*}
where $\gamma = p (1 + 1 / p^{\prime} - \theta - \gamma_2)$. Apply Fubini's theorem followed by the substitution $s^{\prime} = \tau + s$ to get
\begin{align*}
 \int_0^T [\cA \e^{- t \cA} \fu]_{\theta , p}^p \; \d t \leq C \int_0^T \int_0^{\infty} \tau^{\gamma - 1} \int_{T + \tau - t}^{2T + \tau} \frac{1}{s^{\gamma_1 p}} \| \cA \e^{- s / 2 \cA} f(t) \|_X^p \; \d s \; \d \tau \; \d t.
\end{align*}
Finally, use Fubini's theorem in order to calculate the $\tau$-integral (here $\gamma_2 < 1 - \theta + 1 / p^{\prime}$ is essential) and note that $t - T$ is negative and $\gamma$ positive to get
\begin{align*}
 \int_0^T [\cA \e^{- t \cA} \fu]_{\theta , p}^p \; \d t &\leq \frac{C}{\gamma} \int_0^T \int_{T - t}^{\infty} \frac{1}{s^{\gamma_1 p}} \| \cA \e^{- s / 2 \cA} f(t) \|_X^p (s + t - T)^{\gamma} \; \d s \; \d t \\
 &\leq \frac{C}{\gamma} \int_0^T \int_{T - t}^{\infty} s^{\gamma - \gamma_1 p} \| \cA \e^{- s / 2 \cA} f(t) \|_X^p \; \d s \; \d t.
\end{align*}
The proof is concluded by definition $\gamma$ and of the real interpolation space norm, since this gives
\begin{align*}
 \int_0^T [\cA \e^{- t \cA} \fu]_{\theta , p}^p \; \d t \leq \frac{2^{p(1 - \theta)} C}{2 \gamma} \| f \|_{\L^p (0 , T ; D_{\cA} (\theta , p))}^p.
\end{align*}
\begin{flushleft}
\textbf{Step 4.}
\end{flushleft}
In this step, we estimate $\int_0^T \|\cA \e^{-t\cA}\fu\|_X \; \d t$. It is known, see~\cite[Corollary 6.6.3]{haase2006}, that $D_{\cA} (\vartheta , 1) \hookrightarrow D(\cA^{\vartheta})$ and that $D_{\cA}(\theta , p) \hookrightarrow D_{\cA} (\vartheta , 1)$ for every $0 < \vartheta < \theta$. Thus,
\begin{align*}
D_{\cA}(\theta , p) \hookrightarrow D(\cA^{\vartheta}).
\end{align*}
Now, let $\vartheta_1 , \vartheta_2 , \vartheta_3 \in (0,1)$ with $\vartheta_1 + \vartheta_2 + \vartheta_3 = 1$, $\vartheta_1 < \theta$, $\vartheta_2 p^{\prime} < 1$ and $\vartheta_3 p < 1$, where $p^{\prime}$ denotes the H\"older conjugate exponent to $p$. Then, the bounded analyticity of $e^{-t\cA}$, H\"older's inequality and the above embedding imply
\begin{align*}
\| \cA \e^{-t\cA}\fu \|_X &= \| \cA^{\vartheta_3}\e^{-t\cA}\int_0^T \cA^{\vartheta_2} \e^{-(T-s)\cA}\cA^{\vartheta_1} f(s) \; \d s\|_X
\leq C t^{-\vartheta_3}\int_0^T (T-s)^{-\vartheta_2}\|A^{\vartheta_1}f(s)\|_X \; \d s \\
& \leq C t^{-\vartheta_3}\left(\int_0^T (T-s)^{-\vartheta_2 p^{\prime}} \; \d s\right)^{\frac{1}{p^{\prime}}}\left(\int_0^T \|A^{\vartheta_1}f(s)\|_X^p \; \d s\right)^{\frac{1}{p}}\\
&\leq C t^{-\vartheta_3} \| f\|_{\L^p(0,T:D_{\cA}(\theta , p))}.
\end{align*}
Consequently,
\begin{align*}
\int_0^T \|\cA \e^{-t\cA}\fu\|_X \; \d t &\leq c \| f\|_{\L^p(0,T:D_{\cA}(\theta , p))}. \qedhere
\end{align*}
\end{proof}

We conclude this section by showing that, under the assumptions of Theorem~\ref{Thm: Da Prato-Grisvard periodic}, $u$ defined by~\eqref{Eq: Definition of periodic u} indeed is the unique strong solution to~$\mathrm{(PACP)}$.

\begin{proposition}
\label{Prop: Uniqueness of linear problem}
Under the hypotheses of Theorem~\ref{Thm: Da Prato-Grisvard periodic} the function $u$ defined by~\eqref{Eq: Definition of periodic u} is the unique strong solution to~$\mathrm{(PACP)}$, i.e., $u$ is the unique periodic function of period $T$ in $\C(\IR ; X)$ that is for almost every $t \in \IR$ differentiable in $t$, satisfies $u(t) \in D(\cA)$, and $\cA u \in \L^p(0,T;X)$, and $u$ solves
\begin{align*}
 u^{\prime} (t) + \cA u (t) &= f(t).
\end{align*}
\end{proposition}

\begin{proof}
First of all, $u$ is periodic by Lemma~\ref{Lem: Properties of the mild solution} and since $D_{\cA} (\theta , p)$ continuously embeds into $X$ the very same lemma implies $u \in \C(\IR ; X)$. \par
Assume first that $f_{|(0 , T)} \in \L^p(0 , T ; D(\cA))$. Then, by a direct calculation, $u$ defined by~\eqref{Eq: Definition of periodic u} is differentiable, satisfies $u(t) \in D(\cA)$, and solves 
\begin{align*}
 u^{\prime} (t) + \cA u (t) &= f(t)
\end{align*}
 for every $t \in \IR$. The density of $\L^p(0 , T ; D(\cA))$ in $\L^p(0 , T ; D_{\cA} (\theta , p))$ and the estimate proven in Theorem~\ref{Thm: Da Prato-Grisvard periodic} imply that all these properties carry over to all right-hand sides in $\L^p(0 , T ; D_{\cA} (\theta , p))$ (but only for almost every $t \in \IR$) by an approximation argument. \par
 For the uniqueness, assume that $v \in \C(\IR ; X)$ with $v^{\prime}, \cA v \in \L^p(0,T;X)$ is another periodic function of period $T$ which satisfies the equation for almost every $t \in \IR$. 
Let $w := u - v$. Then $w$ satisfies
 \begin{align*}
  w^{\prime} (t) = - \cA w (t) \qquad (\mathrm{a.e.}~ t \in \IR).
 \end{align*}
 In this case, for $t > 0$, $w$ can be written by means of the semigroup as $w (t) = \e^{- t \cA} (u(0) - v(0))$. Now, the exponential decay of the semigroup and the periodicity of $w$ imply that $w$ must be zero for all $t \in \IR$.
\end{proof}

\begin{remark}
\label{Rem: Maximal regularity estimate}
Combining Theorem~\ref{Thm: Da Prato-Grisvard periodic} and Proposition~\ref{Prop: Uniqueness of linear problem} shows that for each periodic $f$ with period $T$ and $f_{|(0 , T)} \in \L^p(0 , T ; D_{\cA} (\theta , p))$ also $u^{\prime}_{|(0 , T)} \in \L^p(0 , T ; D_{\cA} (\theta , p))$. The same is true for $u$ since $0 \in \rho (\cA)$. Summarizing, there exists a constant $C > 0$ such that
\begin{align}
\label{Eq: Maximal regularity estimate}
 \|u\|_{\IE_{\cA}^{\per}} \leq C \| f \|_{\L^p(0 , T ; D_{\cA} (\theta , p))},
\end{align}
where $\IE_{\cA}^{\per}$ is defined as in the end of Section~\ref{Sec: Preliminaries}.
\end{remark}

\section{Time periodic solutions for semilinear equations}
\label{Sec: The periodic solution}

In this section, we use the periodic version of the Da Prato--Grisvard theorem to construct time periodic solutions to semilinear parabolic equations by employing Banach's fixed point theorem. The framework that is presented here includes all the models from Section~\ref{Sec: Main results}. \par

\subsection{An abstract existence theorem for general types of nonlinearities}
\label{Subsec: An abstract existence theorem for general types of nonlinearities}

Let $- \cA$ be the generator of a bounded analytic semigroup $\e^{-t \cA}$ on a Banach space $X$ with the domain $D(\cA)$ and $0 \in \rho (\cA)$. For $T > 0$, $\theta \in (0 , 1)$, and $1 \leq p < \infty$ let $f : \IR \to D_{\cA}(\theta , p)$ be periodic of period $T$ with $f_{|(0 , T)} \in \L^p(0 , T ; D_{\cA} (\theta , p))$. We are aiming for the strong solvability of
\begin{align*} \tag{NACP} \label{Eq: NACP}
 \left\{
  \begin{aligned}
   u^{\prime} (t) + \cA u (t) &= F[u](t) + f(t) \qquad &&(t \in \IR) \\
   u(t) &= u(t + T)  &&(t \in \IR)
  \end{aligned} \right.
\end{align*}
under some smallness assumptions on $f$. The solution $u$ will be constructed in the space of maximal regularity $\IE_{\cA}^{\per}$ defined in the end of Section~\ref{Sec: Preliminaries}. Recall the corresponding data space
\begin{align*}
 \IF_{\cA} = \L^p(0 , T ; D_{\cA} (\theta , p))
\end{align*}
and let $\IB_\rho := \IB^{\IE_{\cA}^{\per}} (0 , \rho)$ for some $\rho > 0$. For the nonlinear term $F$, we make the following standard assumption.

\begin{assumption}{N}
\label{Ass: Nonlinearity}
There exists $R > 0$ such that the nonlinear term $F$ is a mapping from $\IB_R$ into $\IF_{\cA}$ and satisfies
 \begin{align*}
  F \in \C^1(\IB_R ; \IF_{\cA}), \quad F(0) = 0, \quad \text{and} \quad DF(0) = 0, 
 \end{align*}
where $DF : \IB_R \to \cL(\IE_{\cA}^{\per}, \IF_{\cA})$ denotes the Fr$\rm{\acute{e}}$chet derivative.
\end{assumption}

The following theorem proves existence and uniqueness of solutions to $\mathrm{(NACP)}$ in the class $\IE_{\cA}^{\per}$ for small forcings $f$.

\begin{theorem}
\label{Thm: Abstract existence theorem}
Let $T > 0$, $0 < \theta < 1$, $1 \leq p < \infty$, and $F$ and $R > 0$ subject to Assumption~\ref{Ass: Nonlinearity}. 
Then there is a constant $r \leq R$ and $c=c(T , \theta , p , r)>0$ such that if $f : \IR \to D_{\cA} (\theta , p)$ is $T$-periodic with $\|f\|_{\IF_{\cA}} \leq c$, then there exists a unique solution $u : \IR \to D_{\cA}(\theta , p)$ of $\mathrm{(NACP)}$ with the same period $T$ and $u_{|(0,T)} \in \IB_r$. 
\end{theorem}

\begin{proof}
Let $S : \IB_R \to \IE_{\cA}^{\per}, v \mapsto u_v$ be the solution operator of the linear equation 
\begin{align*}
u_v^{\prime}(t) + \cA u_v(t) = F[v(t)] + f(t)\quad {\rm in }\ (0 , T)
\end{align*}
with $u_v (0) = u_v (T)$. 
This is well-defined since $F[v] \in \IF_{\cA}$ by Assumption~\ref{Ass: Nonlinearity}, so that, by Proposition~\ref{Prop: Uniqueness of linear problem} and Remark~\ref{Rem: Maximal regularity estimate}, $u_v$ uniquely exists and lies in $\IE_{\cA}^{\per}$. \par
We prove that this solution operator is a contraction on $\IB_r$ for some $r \leq R$. 
Let $M > 0$ denote the infimum of all constants $C$ satisfying~\eqref{Eq: Maximal regularity estimate}. Choose $r > 0$ small enough such that 
 \begin{align*}
 \sup_{w \in \IB_r} \|DF[w]\|_{\cL(\IE_{\cA}^{\per},\IF_{\cA})} \leq \frac{1}{2M},  
 \end{align*}
which is possible by Assumption~\ref{Ass: Nonlinearity}. 
By virtue of~\eqref{Eq: Maximal regularity estimate} as well as the mean value theorem, estimate for any $v \in \IB_r$ and $f$ satisfying $\|f\|_{\IF_{\cA}} \leq r / (2M) =: c$,
 \begin{align*}
 \|S(v)\|_{\IE_{\cA}^{\per}} & \leq M (\|F[v]\|_{\IF_{\cA}} + \|f\|_{\IF_{\cA}}) \leq M( \sup_{w \in \IB_r} \|DF[w]\|_{\cL(\IE_{\cA}^{\per},\IF_{\cA})}\|v\|_{\IE_{\cA}^{\per}} + \|f\|_{\IF_{\cA}}) \leq r. 
\end{align*}
So $S(\IB_r) \subset \IB_r$. 
Similarly, for any $v_1, v_2 \in \IB_r$, 
\begin{align*}
\|S(v_1) - S(v_2)\|_{\IE_{\cA}^{\per}} \leq M \sup_{w \in \IB_r} \|DF[w]\|_{\cL(\IE_{\cA}^{\per},\IF_{\cA})}\|v_1 - v_2\|_{\IE_{\cA}^{\per}} \leq \frac{1}{2} \|v_1 - v_2\|_{\IE_{\cA}^{\per}}. 
\end{align*}
Consequently, the solution operator $S$ is a contraction on $\IB_r$ and the contraction mapping theorem is applicable. The solution to $\mathrm{(NACP)}$ is defined as follows. Let $u$ be the unique fixed point of $S$. Since $S u = u$, $u$ satisfies $u(0) = u(T)$ and thus can be extended periodically to the whole real line. This function solves $\mathrm{(NACP)}$.
\end{proof}

\subsection{Two special examples}

A short glimpse towards the models presented in Subsections~\ref{Subsec: FitzHugh--Nagumo}-\ref{Subsec: Allen--Cahn} reveals that one of the following situations occurs:
\begin{itemize}
\item The bidomain operator $A$ appears only in the first but not in the second equation of the bidomain models and the nonlinearity depends linearly on the gating variable $w$. (Subsections~\ref{Subsec: FitzHugh--Nagumo}-\ref{Subsec: Rogers--McCulloch})
\item The ODE and the gating variable $w$ are omitted. (Subsection~\ref{Subsec: Allen--Cahn})
\end{itemize}
As a consequence, in the first situation the operator associated with the linearization of the bidomain models can be written as an operator matrix whose first component of the domain embeds into a $\W^{2 , q}$-space. Since the dynamics of the gating variable is described only by an ODE, there appears no smoothing in the spatial variables of $w$. However, as we aim to employ Theorem~\ref{Thm: Abstract existence theorem} and as the nonlinearity of the first equation depends linearly on $w$, at least in the models of Aliev--Panfilov and Rogers--McCulloch, $w$ must be contained in $D_A (\theta , p)$. Otherwise one cannot view the nonlinearity as a suitable right-hand side as it is done in Subsection~\ref{Subsec: An abstract existence theorem for general types of nonlinearities}. Hence, we choose $D_A (\theta , p)$ as the ground space for the gating variable. \par
To describe this situation in our setup, assume in the following, that $- \cA$ is the generator of a bounded analytic semigroup on a Banach space $X = X_1 \times X_2$, with domain $D(\cA) = D(A_1) \times D(A_2)$, and $0 \in \rho(\cA)$. We further set for some $1 < q < \infty$, $1 \leq p < \infty$, and $\theta \in (0 , 1)$
\begin{align*}
 X_1 = \L^q(\Omega) , \quad D(A_1) = D(A), \quad \text{and} \quad X_2 = D(A_2) = D_A (\theta , p).
\end{align*}
Furthermore, define two types of nonlinearities as follows: For $a_1 , a_2 , a_3 , a_4 \in \IR$ let
\begin{align*}
 F_1 [u_1 , u_2] := \begin{pmatrix} a_1 u_1^2 + a_2 u_1^3 + a_3 u_1 u_2 \\ a_4 u_1^2 \end{pmatrix}
\end{align*}
and for $b_1$, $b_2 \in \IR$ let
\begin{align*}
 F_2 [u_1] := b_1 u_1^2 + b_2 u_1^3 .
\end{align*}
Here, $F_1$ will be a prototype of the nonlinearities considered in Subsections~\ref{Subsec: FitzHugh--Nagumo}-\ref{Subsec: Rogers--McCulloch} and $F_2$ for the one considered in Subsection~\ref{Subsec: Allen--Cahn}. For the moment, the condition $0 \in \rho(\cA)$ seems inappropriate as $0 \notin \rho(A)$. However, we will linearize the bidomain equations around suitable stable stationary solutions and in this situation $0 \in \rho(\cA)$ will be achieved. \par
In the following,  we concentrate only on  $F_1$, since  the results for $F_2$ may be proved  in a similar way. To derive conditions on $p$, $q$, and $\theta$ ensuring that $F_1$ 
satisfies Assumption~\ref{Ass: Nonlinearity}, the following two lemmas are essential. The first one is a consequence of the mixed derivative theorem, see, e.g.,~\cite{DHPIII} and reads as follows.

\begin{lemma}\label{Thm: Mixed derivative theorem}
Let $\Omega \subset \IR^n$ be a bounded $\C^2$-domain, $T > 0$, $1 < p , q < \infty$,  and $\sigma \in [0 , 1]$. Then the following continuous embedding is valid 
\begin{align*}
 \W^{1 , p} (0 , T ; \L^q(\Omega)) \cap \L^p (0 , T ; \W^{2 , q} (\Omega)) \subset \W^{\sigma , p} (0 , T ; \W^{2 (1 - \sigma) , q} (\Omega)).
\end{align*}
\end{lemma}

\begin{lemma}
\label{Lem: Multiplication by Besov functions}
Let $\Omega \subset \IR^n$ be a bounded $\C^2$-domain, $1 \leq p < \infty$, $1 < q < \infty$, $q \leq r, s \leq \infty$, $1/r + 1/s = 1/q$, and $\theta \in (0 , 1 / 2)$. Then there exists a constant $C > 0$ such that for 
\begin{align*}
 \| u v \|_{\B^{2 \theta}_{q , p} (\Omega)} \leq C \| u \|_{\W^{1 , s} (\Omega)} \| v \|_{\B^{2 \theta}_{r , p} (\Omega)} \qquad (u \in \W^{1 , s} (\Omega), v \in \B^{2 \theta}_{r , p} (\Omega)).
\end{align*}
\end{lemma}

\begin{proof}
Assume first that $v \in \W^{1 , r} (\Omega)$. By H\"older's inequality it follows that
\begin{align*}
 \| u v \|_{\L^q(\Omega)} \leq \| u \|_{\L^s(\Omega)} \| v \|_{\L^r(\Omega)} \quad \text{and} \quad \| u v \|_{\W^{1 , q}(\Omega)} \leq 2 \| u \|_{\W^{1 , s}(\Omega)} \| v \|_{\W^{1 , r}(\Omega)}.
\end{align*}
Now, real interpolation delivers the desired inequality.
\end{proof}

In the following proposition we elaborate the conditions on $p$, $q$, and $\theta$ that ensure that $F$ maps $\IE_{\cA}^{\per}$ into $\IF_{\cA}$.

\begin{proposition}
Let $1 \leq p < \infty$, $n < q < \infty$ satisfy $1/p + n/(2q) \leq 3/4$ and $\theta \in (0 , 1 / 2)$ there exists a constant $C > 0$ such that
\begin{align*}
 \| F_1 (u_1 , u_2) \|_{\IF_{\cA}} \leq C \big( \| u_1 \|_{\IE_{A_1}^{\per}}^2 + \| u_1 \|_{\IE_{A_1}^{\per}}^3 + \| u_1 \|_{\IE_{A_1}^{\per}} \| u_2 \|_{\IE_{A_2}^{\per}} \big)
\end{align*}
for all $u_1 \in \IE_{A_1}^{\per}$ and $u_2 \in \IE_{A_2}^{\per}$.
\end{proposition}

\begin{proof}
We start with the first component of $F_1$. By \eqref{besov}  we have $D_A (\theta , p) = \B^{2 \theta}_{q , p} (\Omega)$ and Lemma~\ref{Lem: Multiplication by Besov functions} implies
 \begin{align*}
  \| u_1 u_2 \|_{\L^p(0 , T ; D_A (\theta , p))}^p &\leq C \| u_1 u_2 \|_{\L^p(0 , T ; \B^{2 \theta}_{q , p} (\Omega))}^p \leq C \int_0^T \| u_1 \|_{\W^{1 , \infty} (\Omega)}^p \| u_2 \|_{\B^{2 \theta}_{q , p} (\Omega)}^p \; \d t,
\end{align*}
by choosing $r = q$, $s = \infty$ in Lemma \ref{Lem: Multiplication by Besov functions}. 
Using that $\W^{1 , p} (0 , T ; \B^{2 \theta}_{q , p} (\Omega)) \subset \L^{\infty} (0 , T ; \B^{2 \theta}_{q , p} (\Omega))$ delivers
\begin{align*}
 \| u_1 u_2 \|_{\L^p(0 , T ; D_A (\theta , p))}^p \leq C \| u_2 \|_{\W^{1 , p} (0 , T ; \B^{2 \theta}_{q , p} (\Omega))}^p \| u_1 \|_{\L^p(0 , T ; \W^{1 , \infty} (\Omega))}^p.
 \end{align*}
 Finally, note that $D(A_1) \subset \W^{2 , q} (\Omega) \subset \W^{1 , \infty} (\Omega)$ if $n < q$. Next, by the continuous embedding 
$\W^{1 , q} (\Omega) \subset \B^{2 \theta}_{ q , p} (\Omega)$, H\"older's inequality and the mixed derivative theorem, we obtain  for $\alpha \in \{ 2 , 3 \}$
 \begin{align*}
  \| u_1^{\alpha} \|_{\L^p (0 , T ; D_A (\theta , p))} &\leq C \| u_1 \|_{\L^{\alpha p} (0 , T ; \W^{1 , \alpha q} (\Omega))}^{\alpha} 
  \leq C \| u_1 \|_{\W^{\sigma , p} (0 , T ; \W^{2(1- \sigma), q} (\Omega))}^{\alpha}.
 \end{align*}
 provided $\sigma \in [0,1]$ satisfies 
 \begin{align*}
\sigma - 1/p \geq - 1/(\alpha p),~~~{\rm and}~~~2(1-\sigma) - n/q \geq 1- n/(\alpha q). 
 \end{align*}
 The condition $1/p + n/(2q) \leq 3/4$ guarantees the existence of $\sigma$ for $\alpha \in \{ 2 , 3 \}$. The second component of $F_1$ was already estimated above. 
 \end{proof}

Finally, by definition of $F_1$ it is clear that $F_1 (0 , 0) = 0$. Moreover, due to the polynomial structure of $F_1$ it is clear that $F_1$ is Fr\'echet differentiable 
with $D F_1 (0 , 0) = 0$. Hence, we have the following proposition.

\begin{proposition}\label{Prop: Nonlinearities satisfy assumptions}
With the definitions of this subsection the nonlinearities $F_1$ and $F_2$ satisfy Assumption~\ref{Ass: Nonlinearity}.
\end{proposition}

\section{Proofs of the Main Theorems}\label{Sec: Examples}

Before treating the models described in Section~\ref{Sec: Main results}, we remark that the linear part of the bidomain systems will be represented as an  operator matrix  
and it will be eminent that the negative of this operator matrix generates a bounded analytic semigroup. This will be proven in the following lemma.

\begin{lemma}
\label{Lem: operator matrix}
Let $- B$ be the generator of a bounded analytic semigroup on a Banach space $X_1$ with $0 \in \rho(B)$, $1 \leq p < \infty$, and $\theta \in (0 , 1)$. Let $X_2 = D_B (\theta , p)$ and define for $d > 0$ and $b , c \geq 0$ the operator $\cA : X := X_1 \times X_2 \to X$ with domain $D (\cA) := D(B) \times X_2$ by
 \begin{align*}
  \cA := \begin{pmatrix} B & b \\ - c & d \end{pmatrix}.
 \end{align*}
Then $- \cA$ generates a bounded analytic semigroup on $X$ with $0 \in \rho(\cA)$.
\end{lemma}

\begin{proof}
Let $\Sigma_{\omega}$, $\omega \in (\pi / 2 , \pi]$, be a sector that satisfies $\rho(- B) \subset \Sigma_{\omega}$ with
\begin{align*}
 \| \lambda (\lambda + B)^{-1} \|_{\Lop(X_1)} \leq C \qquad (\lambda \in \Sigma_{\omega}).
\end{align*}
First note that $0 \in \rho(\cA)$; its inverse being
\begin{align*}
 \cA^{-1} = \begin{pmatrix} d & - b \\ c & B \end{pmatrix} (b c + d B)^{-1}.
\end{align*}
Note that the choice $X_2 = D_B (\theta , p)$ is used here as $\cA^{-1}$ is only an operator from $X_1 \times X_2$ onto $D(B) \times X_2$ if $D(B) \subset X_2 \subset X_1$ and if $B (b c + d B)^{-1}$ maps $X_2$ into $X_2$. By the definition of $D_B (\theta , p)$ in~\eqref{Eq: Characterization real interpolation space} this latter is satisfied. \par
For the resolvent problem let $\lambda \in \Sigma_{\beta}$, $\beta \in (\pi / 2 , \omega)$ to be chosen. Then,
\begin{align*}
 (\lambda + \cA)^{-1} = (\lambda + d)^{-1} \begin{pmatrix} \lambda + d & -b \\ c & \lambda + B \end{pmatrix} \Big( \lambda + \frac{bc}{\lambda + d} + B \Big)^{-1}
\end{align*}
whenever $\lambda + \frac{bc}{\lambda + d} \in \rho(- B)$. To determine the angle $\beta$ for which $\lambda + \frac{bc}{\lambda + d} \in \rho(- B)$ distinguish between the cases $\lvert \lambda \rvert < M$ and $\lvert \lambda \rvert \geq M$ for some suitable constant $M > 0$. Notice that only the case $b , c > 0$ is of interest. Let $C_{\omega} > 0$ be a constant depending solely on $\omega$ such that $\lvert \lambda + d \rvert \geq C_{\omega} (\lvert \lambda \rvert + d)$. Choose $M$ such that $\lvert \lambda \rvert \geq M$ if and only if
\begin{align}
\label{Eq: Defining condition for M}
 C_{\omega} \sin(\omega - \beta) [ \lvert \lambda \rvert^2 + d \lvert \lambda \rvert ] \geq 2 b c.
\end{align}
 This implies
\begin{align*}
 \Big\lvert \frac{b c}{\lambda + d} \Big\rvert \leq \frac{b c}{C_{\omega} (\lvert \lambda \rvert + d)} \leq \frac{\lvert \lambda \rvert \sin(\omega - \beta)}{2}
\end{align*}
and thus that $\lambda + \frac{b c}{\lambda + d} \in \Sigma_{\omega}$. Moreover,
\begin{align}
\label{Eq: Comparison with lambda}
 \Big\lvert \lambda + \frac{bc}{d + \lambda} \Big\rvert \geq \lvert \lambda \rvert \Big( 1 - \frac{\sin (\omega - \beta)}{2} \Big).
\end{align}
Next, choose $\beta$ that close to $\pi / 2$ such that
\begin{align}
\label{Eq: Small lambda}
 M \sin(\beta - \pi / 2) \leq \frac{b c d}{b c + (d + M)^2}. 
\end{align}
Notice that $M$ itself depends on $\beta$, however, it depends only uniformly on its distance to $\omega$ by~\eqref{Eq: Defining condition for M}. In the case $\lvert \lambda \rvert < M$ the validity of~\eqref{Eq: Small lambda} together with trigonometric considerations implies that $\Re\big( \lambda + \frac{bc}{d + \lambda} \big) \geq 0$ proving that under conditions~\eqref{Eq: Defining condition for M} and~\eqref{Eq: Small lambda} we have $\lambda + \frac{bc}{d + \lambda} \in \Sigma_{\omega}$ whenever $\lambda \in \Sigma_{\beta}$. We conclude that $\lambda \in \rho(- \cA)$. To obtain the resolvent estimate, we calculate
\begin{align*}
 \| \lambda (\lambda + \cA)^{-1} &\|_{\Lop(X)} \leq \Big\| \lambda \Big( \lambda + \frac{bc}{\lambda + d} + B \Big)^{-1} \Big\|_{\Lop(X_1)} + \Big\lvert \frac{\lambda b}{\lambda + d} \Big\rvert \Big\| \Big( \lambda + \frac{bc}{\lambda + d} + B \Big)^{-1} \Big\|_{\Lop(X_2 , X_1)} \\
 & + \Big\lvert \frac{\lambda c}{\lambda + d} \Big\rvert \Big\| \Big( \lambda + \frac{bc}{\lambda + d} + B \Big)^{-1} \Big\|_{\Lop(X_1 , X_2)} + 
\Big \lvert \frac{\lambda}{\lambda + d} \Big\rvert \Big\| (\lambda + B) \Big( \lambda + \frac{bc}{\lambda + d} + B \Big)^{-1} \Big\|_{\Lop(X_2)}. 
\end{align*}
The first term on the right-hand side is directly handled by the resolvent estimate of $B$. The second is treated by this resolvent estimate as well and by noting that $X_2 \subset X_1$. The fourth term is estimated by using that the definition of $X_2$ in~\eqref{Eq: Characterization real interpolation space} implies resolvent estimates in $X_2$ (the resolvent commutes with the semigroup appearing in~\eqref{Eq: Characterization real interpolation space}). For the third term, the estimate follows from the invertibility of $B$ and the interpolation inequality $\| x \|_{X_2} \leq C \| x \|_{X_1}^{1 - \theta} \| B x \|_{X_1}^{\theta}$. Altogether, this yields
\begin{align*}
 \| \lambda (\lambda + \cA)^{-1} \|_{\Lop(X)} &\leq C \bigg( \lvert \lambda \rvert + \Big\lvert \frac{\lambda b}{\lambda + d} \Big\rvert + \Big\lvert \frac{\lambda c}{\lambda + d} \Big\rvert \Big\lvert \lambda + \frac{b c}{\lambda + d} \Big\rvert^{\theta} + \Big\lvert \frac{\lambda^2}{\lambda + d} \Big\rvert \bigg) \Big\lvert \lambda + \frac{b c}{\lambda + d} \Big\rvert^{-1} 
 + C \Big\lvert \frac{\lambda}{\lambda + d} \Big\rvert.
\end{align*}
The resolvent estimate for $\lvert \lambda \rvert \geq M$ follows by means of the uniform boundedness of the term $\lvert \lambda / (\lambda + d) \rvert$ and by~\eqref{Eq: Comparison with lambda}. \par
For $\lvert \lambda \rvert< M$ the function $\lambda \mapsto \lambda (\lambda + \cA)^{-1}$ is continuous on $\overline{\Sigma_{\beta}} \cap \overline{B(0 , M)}$ since $0 \in \rho(\cA)$. This implies the resolvent estimate also for small $\lambda$.
\end{proof}

Now, we are ready to prove the main results presented in Section~\ref{Sec: Main results}. To do so, the equilibrium points of the nonlinearities are calculated for the respective models. Afterwards, the solutions to the bidomain models are written as the sum of the equilibrium solution and a perturbation. This results in an equation for the perturbation which is shown via Theorem~\ref{Thm: Abstract existence theorem} to have strong periodic solutions for suitable equilibrium points.

\subsection{The periodic bidomain FitzHugh--Nagumo equation}

Recall the periodic bidomain FitzHugh--Nagumo equation
\begin{align}
\left\{
\begin{aligned}
\partial_t u+\varepsilon Au&=I-\frac{1}{\varepsilon}[u^3-(a+1)u^2+au+w] \qquad&{\rm in}\ \IR\times \Omega,\\
\partial_t w &=cu-bw &{\rm in}\ \IR\times \Omega,\\
u(t) &= u(t+T) &{\rm in} \ \IR \times \Omega,\\
w(t) &= w(t+T) &{\rm in} \ \IR \times \Omega.
\end{aligned}\label{Eq: bfhn2}
\right.
\end{align}

In order to calculate the equilibrium points, we consider
\begin{align}
u^3-(a+1)u^2+au+w&=0, \label{Eq: fhnes1}\\
 cu-bw &=0. \label{Eq: fhnes2}
\end{align}
Then, the equilibrium points are $(u_1,w_1)=(0,0)$ and assuming $c<b\big(\frac{(a+1)^2}{4}-a\big)$, we obtain furthermore
\begin{align}
(u_2,w_2)&=\left(\frac12(a+1-d),\frac{c}{2b}(a+1-d)\right),\label{Eq: fhnep2.2}\\
(u_3,w_3)&=\left(\frac12(a+1 + d),\frac{c}{2b} (a+1+d)\right),\label{Eq: fhnep3.2}
\end{align}
with $d=\sqrt{(a+1)^2-4(a+\frac{c}{b})}$.
In the following, we use the results from Sections~\ref{Sec: Da Prato--Grisvard} and~\ref{Sec: The periodic solution} to obtain periodic solutions in a neighborhood of these equilibrium points. For this purpose, we use Taylor expansion at the equilibrium points and perform the following change of variables 
\begin{align*}
\begin{pmatrix}v\\ z\end{pmatrix}:=\begin{pmatrix}u-u_i \\ w-w_i\end{pmatrix}
\end{align*}
for $i=1,2,3$. Then, functions $F$ and $G$ describing the ionic transport defined as in the introduction read as follows
\begin{align*}
F(v , z) &= \frac{1}{\varepsilon}[v^3 + (3u_i - a -1)v^2 + (3 u_i^2 -2(a+1)u_i +a)v + z],\\
G(v , z) &= -cv + bz.
\end{align*}
Plugging this into equation~\eqref{Eq: bfhn2} and shifting the linear parts of $F$ and $G$ to the left-hand side yields
\begin{align}
\left\{
\begin{aligned}
\partial_t \begin{pmatrix}v \\ z\end{pmatrix}+\begin{pmatrix}\varepsilon A+\frac{1}{\varepsilon}[3u_i^2-2(a+1)u_i+a]& \frac{1}{\varepsilon} \\ -c &b\end{pmatrix}\begin{pmatrix}v \\ z\end{pmatrix}&=\begin{pmatrix}I-\frac{1}{\varepsilon} [v^3+(3u_i-a-1)v^2]\\ 0\end{pmatrix},\\
v(t)&=v(t+T),\\
z(t)&=z(t+T).
\end{aligned}\label{Eq: bfhn shifted}
\right.
\end{align}


First of all, notice that Proposition~\ref{Prop: Nonlinearities satisfy assumptions} implies that the nonlinearity in~\eqref{Eq: bfhn shifted} satisfies Assumption~\ref{Ass: Nonlinearity}. Next, regarding the system with respect to the equilibrium point $(0 , 0)$, then $-(\varepsilon A +\frac{a}{\varepsilon})$ generates a bounded analytic semigroup by Proposition~\ref{Prop: Giga-Kajiwara} and since $0 \in \rho(\varepsilon A +\frac{a}{\varepsilon})$, we may apply 
Lemma~\ref{Lem: operator matrix} to conclude that the negative of the operator matrix in~\eqref{Eq: bfhn shifted} has zero in its resolvent set and generates a bounded analytic semigroup. Consequently, Theorem~\ref{Thm: Abstract existence theorem} is applicable in the case of the equilibrium point $(0 , 0)$ and delivers a unique strong periodic solution $(v , z)$ to~\eqref{Eq: bfhn shifted} in the desired function space for small periodic forcings $I$. \par
For the second equilibrium point we have $3u_2^2-2(a+1)u_2+a<0$. Since $0 \in \sigma(A)$ the operator $- (\varepsilon A+\frac{1}{\varepsilon}[3u_2^2-2(a+1)u_2+a])$ does not generate a bounded analytic semigroup so that Lemma~\ref{Lem: operator matrix} is not applicable. \par
If
\begin{align*}
u_3>\frac{a+1+\sqrt{(a+1)^2-3a}}{3},
\end{align*} 
we obtain $3u_3^2-2(a+1)u_3+a>0$. Thus, $-(\varepsilon A +\frac{1}{\varepsilon}[3u_3^2-2(a+1)u_3+a])$ generates a bounded analytic semigroup by Proposition~\ref{Prop: Giga-Kajiwara} and  $0\in \rho (\varepsilon A+\frac{1}{\varepsilon} [3u_3^2-2(a+1)u_3+a])$. Hence, we can apply Lemma~\ref{Lem: operator matrix} to conclude that the negative of the operator matrix in~\eqref{Eq: bfhn shifted} has zero in its resolvent set and generates a bounded analytic semigroup. Consequently, Theorem~\ref{Thm: Abstract existence theorem} is applicable in this case of the equilibrium point $(u_3 , w_3)$ and delivers a unique strong periodic solution $(v , z)$ to~\eqref{Eq: bfhn shifted} in the desired function spaces for small periodic forcings $I$. This proves Theorem~\ref{Thm: FitzHugh-Nagumo}.


\subsection{The periodic bidomain Aliev--Panfilov equation}
Recall the periodic bidomain Aliev--Panfilov equation
\begin{align}
\left\{
\begin{aligned}
\partial_t u+ \varepsilon Au&=I-\frac{1}{\varepsilon}[ku^3-k(a+1)u^2+kau+uw] \qquad &{\rm in}\ \IR\times \Omega,\\
\partial_t w &= -(ku(u-1-a)+dw) &{\rm in}\ \IR\times \Omega,\\
u(t) &= u(t+T) &{\rm in} \ \IR \times \Omega,\\
w(t) &= w(t+T) &{\rm in} \ \IR \times \Omega.
\end{aligned} \label{Eq: bap2}
\right.
\end{align}
In order to calculate the equilibrium points, we consider
\begin{align}
ku^3-k(a+1)u^2+kau+uw&=0, \label{Eq: apes1}\\
ku(u-1-a)+dw&=0. \label{Eq: apes2}
\end{align}
Then, the equilibrium points are $(u_1,w_1)=(0,0)$ and, if we assume $\frac{(a+1)^2}{4} + \frac{da}{1-d} > 0$, furthermore
\begin{align}
(u_2 , w_2) &= \left(\frac{a+1}{2}-e, -k u_2^2 + k(a+1)u_2-ka\right),\label{Eq: apep2.2}\\
(u_3 , w_3) &= \left(\frac{a+1}{2}+e, -k u_3^2 + k(a+1)u_3-ka\right)\label{Eq: apep3.3}.
\end{align}
with $e=\sqrt{\frac{(a+1)^2}{4}+\frac{da}{1-d}}$. In the following, we want to use the results from Sections~\ref{Sec: Da Prato--Grisvard} and~\ref{Sec: The periodic solution} to obtain periodic solutions in a neighborhood of these equilibrium points. For this purpose, we use Taylor expansion at the equilibrium points and perform the following change of variables 
\begin{align*}
\begin{pmatrix}v\\ z\end{pmatrix}:=\begin{pmatrix}u-u_i \\ w-w_i\end{pmatrix}
\end{align*}
for $i=1,2,3$. Then, functions $F$ and $G$ describing the ionic transport defined as in the introduction read as follows
\begin{align*}
F(v , z) &= \frac{1}{\varepsilon} [kv^3+(3ku_i-k(a+1))v^2+(3ku_i^2-2k(a+1)u_i+ka+w_i)v+u_iz+vz],\\
G(v , z) &= (2ku_i - k(a+1))v +dz + kv^2.
\end{align*}

Plugging this into equation~\eqref{Eq: bap2} and shifting the linear parts of $F$ and $G$ to the left-hand side yields
\begin{align}
\left\{
\begin{aligned}
\partial_t \begin{pmatrix}v \\ z\end{pmatrix}+&\begin{pmatrix}\varepsilon A+\frac{1}{\varepsilon} [3ku_i^2-2k(a+1)u_i+ka+w_i]& \frac{u_i}{\varepsilon} \\ 2ku_i - k(a+1) &d\end{pmatrix}\begin{pmatrix}v \\ z\end{pmatrix}\\
&=\begin{pmatrix}I-\frac{1}{\varepsilon} [kv^3+(3ku_i-k(a+1))v^2+vz]\\ - kv^2\end{pmatrix},\\
v(t)&=v(t+T),\\
z(t)&=z(t+T).
\end{aligned}\label{Eq: bap shifted}
\right.
\end{align}


According to Proposition~\ref{Prop: Nonlinearities satisfy assumptions}, the nonlinearity in~\eqref{Eq: bap shifted} satisfies Assumption~\ref{Ass: Nonlinearity}. Moreover, considering the system for the equilibrium point $(0 , 0)$, then $-(\varepsilon A +\frac{ka}{\varepsilon})$ generates a bounded analytic semigroup by Proposition~\ref{Prop: Giga-Kajiwara} and since $0 \in \rho(\varepsilon A +\frac{ka}{\varepsilon})$, we can apply Lemma~\ref{Lem: operator matrix} to conclude that the negative of the operator matrix in~\eqref{Eq: bap shifted} has zero in its resolvent set and generates a bounded analytic semigroup. Consequently, Theorem~\ref{Thm: Abstract existence theorem} is applicable in the case of the equilibrium point $(0 , 0)$ and delivers a unique strong periodic solution $(v , z)$ to~\eqref{Eq: bap shifted} in the desired function space for small periodic forcings $I$.\par
For the second equilibrium point we see that $u_2 < 0$, so that the component in the upper right component of the operator matrix is negative. Therefore, we cannot apply Lemma~\ref{Lem: operator matrix} for $(u_2,w_2)$.\par
Similarly, for $(u_3 , w_3)$ it is
\begin{align*}
2ku_3 - k(a+1) = 2 k e > 0.
\end{align*}
Hence, Lemma~\ref{Lem: operator matrix} is not applicable in this case. Altogether, Theorem~\ref{Thm: Aliev-Panfilov} follows.


\subsection{The periodic bidomain Rogers--McCulloch equation}

Recall the periodic bidomain Rogers--McCulloch equation
\begin{align}
\left\{
\begin{aligned} 
\partial_t u+\varepsilon Au&=I-\frac{1}{\varepsilon}[bu^3-b(a+1)u^2+bau+uw] \qquad &{\rm in}\ \IR\times \Omega,\\
\partial_t w &= cu-dw \qquad &{\rm in}\ \IR\times \Omega,\\
u(t) &= u(t+T) &{\rm in} \ \IR \times \Omega,\\
w(t) &= w(t+T) &{\rm in} \ \IR \times \Omega.
\end{aligned} \label{Eq: bmc2}
\right.
\end{align}

In order to calculate the equilibrium points, we consider
\begin{align}
bu^3-b(a+1)u^2+bau+uw&=0, \label{Eq: mces1}\\
 cu-dw&=0. \label{Eq: mces2}
\end{align}
Then, the equilibrium points are $(u_1,w_1)=(0,0)$ and, if we assume $\left(a+1-\frac{c}{bd}\right)^2-4a >0$, furthermore
\begin{align}
(u_2,w_2)&=\left(\frac12 \Big(a+1-\frac{c}{bd}-e\Big),\frac{c}{2d}\cdot \Big(a+1-\frac{c}{bd}-e\Big)\right),\label{Eq: mcep2.2}\\
(u_3,w_3)&=\left(\frac12 \Big(a+1-\frac{c}{bd}+e\Big),\frac{c}{2d}\cdot \Big(a+1-\frac{c}{bd}+e\Big)\right).\label{Eq: mcep3.2}
\end{align}
with $e=\sqrt{\left(a+1-\frac{c}{bd}\right)^2-4a}$.  In the following, we want to use the results from Sections~\ref{Sec: Da Prato--Grisvard} and~\ref{Sec: The periodic solution} to obtain periodic solutions in a neighborhood of these equilibrium points. For this purpose, we use Taylor expansion at the equilibrium points and perform the following change of variables 
\begin{align*}
\begin{pmatrix}v\\ z\end{pmatrix}:=\begin{pmatrix}u-u_i \\ w-w_i\end{pmatrix}
\end{align*}
for $i=1,2,3$. Then, functions $F$ and $G$ describing the ionic transport defined as in Section~\ref{Sec: Introduction} read as follows
\begin{align*}
F(v , z) &= \frac{1}{\varepsilon} [bv^3+(3bu_i-b(a+1))v^2+(3bu_i^2-2b(a+1)u_i+ba+w_i)v+u_iz+vz],\\
G(v , y) &= -cv + dz.
\end{align*}
Plugging this into equation~\eqref{Eq: bmc2} and shifting the linear parts of $F$ and $G$ to the left-hand side yields
\begin{align}
\left\{
\begin{aligned}
\partial_t \begin{pmatrix}v \\ z\end{pmatrix}+&\begin{pmatrix}\varepsilon A+\frac{1}{\varepsilon} [3bu_i^2-2b(a+1)u_i+ba+w_i]& \frac{u_i}{\varepsilon} \\ -c &d\end{pmatrix}\begin{pmatrix}v \\ z\end{pmatrix}\\
&=\begin{pmatrix}I-\frac{1}{\varepsilon} [bv^3+(3bu_i-b(a+1))v^2+vz]\\ 0\end{pmatrix},\\
v(t)&=v(t+T),\\
z(t)&=z(t+T).
\end{aligned}\label{Eq: bmc shifted}
\right.
\end{align}



According to Proposition~\ref{Prop: Nonlinearities satisfy assumptions}, the nonlinearity in~\eqref{Eq: bmc shifted} satisfies Assumption~\ref{Ass: Nonlinearity}. Next, considering the equilibrium point $(0 , 0)$, the operator $-(\varepsilon A +\frac{ba}{\varepsilon})$ generates a bounded analytic semigroup by Proposition~\ref{Prop: Giga-Kajiwara} and since $0\in\rho(\varepsilon A +\frac{ba}{\varepsilon})$, we can apply Lemma~\ref{Lem: operator matrix} to conclude that the negative of the operator matrix in~\eqref{Eq: bmc shifted} has zero in its resolvent set and generates a bounded analytic semigroup. Consequently, Theorem~\ref{Thm: Abstract existence theorem} is applicable in the case of the equilibrium point $(0 , 0)$ and delivers a unique strong periodic solution $(v,z)$ to~\eqref{Eq: bmc shifted} in the desired function space for small forcings $I$. \par
Next, equation~\eqref{Eq: mces1} implies $w_i = -bu_i^2 + b(a+1)u_i - ba$ for $i=2,3$. Then
\begin{align*}
3bu_i^2-2b(a+1)u_i+ba+w_i = u_i (2bu_i - b(a+1)).
\end{align*} 
Hence, for the second equilibrium point we either have $3bu_2^2-2b(a+1)u_2+ba+w_2 <0$, then $-(\varepsilon A+\frac{1}{\varepsilon} [3bu_2^2-2b(a+1)u_2+ba+w_2])$ does not generate a bounded analytic semigroup, or $u_2 <0$. Therefore, we cannot apply Lemma~\ref{Lem: operator matrix} for $(u_2,w_2)$.\par
If we assume 
\begin{align*}
\sqrt{\big(a+1-\frac{c}{bd}\big)^2-4a}-\frac{c}{bd}>0,
\end{align*}
we obtain $3bu_3^2-2b(a+1)u_3+ba+w_3>0$ and $u_3>0$. Thus, $-(\varepsilon A+\frac{1}{\varepsilon} [3bu_3^2-2b(a+1)u_3+ba+w_3])$ generates a bounded analytic semigroup by Proposition~\ref{Prop: Giga-Kajiwara} and $0 \in \rho (\varepsilon A+\frac{1}{\varepsilon} [3bu_3^2-2b(a+1)u_3+ba+w_3])$. Hence, we can apply Lemma~\ref{Lem: operator matrix} to conclude that the negative of the operator matrix in~\eqref{Eq: bmc shifted} has zero in its resolvent and generates a bounded analytic semigroup. Thus, Theorem~\ref{Thm: Abstract existence theorem} is applicable in this case for $(u_3, w_3)$ and delivers a unique strong periodic solution $(v,z)$ in the desired function space for small forcings $I$. This delivers Theorem~\ref{Thm: Rogers-McCulloch}.

\subsection{The periodic bidomain Allen--Cahn equation}

Recall the periodic bidomain Allen--Cahn equation
\begin{align}
\left\{
\begin{aligned}
\partial_t u +Au&=I+u-u^3 \qquad &{\rm in}\ \IR\times \Omega,\\
u(t) &= u(t+T) &{\rm in} \ \IR \times \Omega.
\end{aligned} \label{Eq: bac2}
\right.
\end{align} 
The equilibrium points of this system are $u_1=-1$, $u_2=0$, and $u_3=1$. In the following, we want to use the results from Sections~\ref{Sec: Da Prato--Grisvard} and~\ref{Sec: The periodic solution} to obtain periodic solutions in a neighborhood of these equilibrium points. For this purpose, we use Taylor expansion at the equilibrium points and perform the change of variables $v = u - u_i$ for $i=1,2,3$. Then, the function $F(u)= u^3-u$ reads as follows
\begin{align*}
F(v)= v^3 + 3 u_i v^2 -(1 - 3 u_i^2)v, \qquad i=1,2,3.
\end{align*}
Plugging this into equation~\eqref{Eq: bac2} and shifting the linear parts of $F$ to the left-hand side yields 
\begin{align}
\left\{
\begin{aligned}
\partial_t v +(A - 1 + 3u_i^2)v&=I-v^3 - 3 u_i v^2 \qquad &&{\rm in}\ \IR\times \Omega,\\
u(t) &= u(t+T) &&{\rm in} \ \IR \times \Omega
\end{aligned} \label{Eq: bac shifted}
\right.
\end{align} 
for $i=1,2,3$. According to Proposition~\ref{Prop: Nonlinearities satisfy assumptions}, the nonlinearity in~\eqref{Eq: bac shifted} satisfies Assumption~\ref{Ass: Nonlinearity}. Since $-(A+2)$ generates a bounded analytic semigroup by Proposition~\ref{Prop: Giga-Kajiwara} and since $0 \in \rho (A+2)$, Theorem~\ref{Thm: Abstract existence theorem} is applicable in the case of the equilibrium points $u_1$ and $u_3$ and delivers a unique strong periodic solution $v$ to~\eqref{Eq: bac shifted} in the desired function space for small forcings $I$. Thus, we obtain Theorem~\ref{Thm: Allen-Cahn}.

\nocite{*}


\end{document}